\documentclass[12pt]{amsart}

\usepackage{amssymb}
\usepackage{bm}
\usepackage{graphicx}
\usepackage{psfrag}
\usepackage[usenames,dvipsnames]{color}
\usepackage{enumerate}
\usepackage{url}

\usepackage{algpseudocode}

\usepackage{mathtools}

\usepackage{xy}
\input xy
\xyoption{all}

\newcommand{\excise}[1]{}

\newtheorem{thm}{Theorem}[section]
\newtheorem{lemma}[thm]{Lemma}

\newtheorem{prop}[thm]{Proposition}
\newtheorem{conj}[thm]{Conjecture}

\theoremstyle{definition}

\numberwithin{equation}{section}

\newcommand{\ring}[1]{\ensuremath{\mathbb{#1}}}

\renewcommand\>{\rangle}
\newcommand\<{\langle}

\newcommand\ZZ{\ring{Z}}

\newcommand\kk{\Bbbk}

\DeclareMathOperator\Hom{Hom} 
\DeclareMathOperator\lcm{lcm} 

\newcommand{\frob}{\mathsf{F}}

\begin{document}

\mbox{}
\title[Families of numerical semigroups and the Huneke-Wiegand conjecture]{Families of numerical semigroups and a special case of the Huneke-Wiegand conjecture}

\author[Landeros]{Miguel Landeros}
\address{Mathematics Department \\ University of Iowa }
\email{mlanderos314@gmail.com}

\author[O'Neill]{Christopher O'Neill}
\address{Mathematics Department\\San Diego State University\\San Diego, CA 92182}
\email{cdoneill@sdsu.edu}

\author[Pelayo]{Roberto Pelayo}
\address{Mathematics Department\\University of California, Irvine\\Irvine, CA 92697}
\email{rcpelayo@uci.edu}

\author[Pe\~{n}a]{Karina Pe\~{n}a}

\author[Ren]{James Ren}
\address{Mathematics \& Computer Science Department \\ University of Guam \\ Mangilao, Guam 96913}
\email{renj@triton.uog.edu}

\author[Wissman]{Brian Wissman}
\address{Mathematics Department\\University of Hawai`i at Hilo\\Hilo, HI 96720}
\email{wissman@hawaii.edu}



\date{\today}

\begin{abstract}
The Huneke-Wiegand conjecture is a decades-long open question in commutative algebra.  Garc\'ia-S\'anchez and Leamer showed that a special case of this conjecture concerning numerical semigroup rings $\kk[\Gamma]$ can be answered in the affirmative by locating certain arithmetic sequences within the numerical semigroup~$\Gamma$.  In this paper, we use their approach to prove the Huneke-Wiegand conjecture in the case where $\Gamma$ is generated by a generalized arithmetic sequence and showcase how visualizations can be leveraged to find the requisite arithmetic sequences.  
\end{abstract}

\maketitle


\section{Introduction}
\label{sec:intro}

Numerical semigroups, co-finite additive subsemigroups of the natural numbers, have long been studied for their relationship to important objects in commutative algebra.  Given a numerical semigroup $\Gamma \subseteq \ZZ_{\ge 0}$ generated by $n_1, \ldots, n_k$, which we denote
$$\Gamma = \langle n_1, n_2, \ldots, n_k\rangle = \{z_1n_1 + z_2n_2 + \cdots + z_k n_k : z_i \in \ZZ_{\ge 0}\},$$
the semigroup algebra $\kk[\Gamma] = \kk[x^{n_1}, \ldots, x^{n_k}]$ over a field $\kk$ is the subring of the polynomial ring $\kk[x]$ for which every term $x^n$ appearing with nonzero coefficient in an element of $\kk[\Gamma]$ has $n \in \Gamma$.  
Understanding monomial ideals in this ring, which is inherently a problem in commutative algebra, can be attacked by studying the underlying semigroup $\Gamma$.  The advantages of this approach are manifold, as numerical semigroups have a well-studied factorization theory~\cite{numericalsurvey,numerical} and several computational packages \cite{numericalsgpsgap,numericalsgpssage}. 

One specific open problem that has benefited explicitly from this relationship is the following special case of the Huneke-Wiegand conjecture~\cite{hwintersection,hwconj}.  

\begin{conj}\label{conj:hw}
If $R = \kk[\Gamma]$ is the semigroup algebra of a symmetric numerical semigroup $\Gamma$, and $M$ is a 2-generated ideal of $R$, viewed as a module over $\kk[\Gamma]$, then the torsion submodule of $M \otimes_{R} \Hom_R(M, R)$ is non-trivial.
\end{conj}

In its general form, the Huneke-Wiegand conjecture~\cite{hwconj}, which has been open for 3 decades, concerns a one-dimensional Gorenstein domain $R$ and a finitely generated $R$-module $M$ that is not projective.  If $R = \kk[\Gamma]$ is a numerical semigroup algebra, then $R$ is one-dimensional, and the Gorenstein hypothesis on $R$ is equivalent to $\Gamma$ being \emph{symmetric} (that is, for every $n \not \in \Gamma$, we have $F(S) - n \in \Gamma$, where $F(\Gamma) = \max(\ZZ_{\ge 0} \setminus \Gamma)$ is the \emph{Frobenius number} of $\Gamma$).  

In~\cite{hwintersection}, Garc\'ia-S\'anchez and Leamer showed that Conjecture~\ref{conj:hw} can be positively answered for a given numerical semigroup algebra $R = \kk[\Gamma]$ if certain irreducible arithmetic sequences can be found inside $\Gamma$ itself.  
Fix a positive $s \notin \Gamma$ and let
$$S_\Gamma^s = \{(n,\ell) : n, \ell \in \ZZ_{\ge 1} \text{ with } n, n + s, \ldots, n + \ell s \in \Gamma\}$$
encode the arithmetic sequences of step size $s$ that are contained in $\Gamma$.  
Note that $S_\Gamma^s$ is closed under component-wise addition since $(n_1,\ell_1), (n_2,\ell_2) \in S_\Gamma^s$ implies
$$n_1 + n_2, n_1 + n_2 + s, \ldots, n_1 + n_2 + (\ell_1 + \ell_2)s \in \Gamma.$$
The authors of~\cite{hwintersection} proved that Conjecture~\ref{conj:hw} holds if, for any numerical semigroup~$\Gamma$ and any positive $s \notin \Gamma$, some element $(n, 2) \in S_\Gamma^s$ is \emph{irreducible}, meaning it cannot be written as a sum of other elements of $S_\Gamma^s$.  
This result has been leveraged to verify Conjecture~\ref{conj:hw} for some well-studied families of numerical semigroups, such as when $\Gamma$ is a complete intersection~\cite{hwintersection} or when $\Gamma$ is generated by an arithmetic sequence whose step size coincides with $s$~\cite{teambob}.  


In this paper, we utilize the results of~\cite{hwintersection} to prove Conjecture~\ref{conj:hw} whenever $\Gamma$ is generated by a \emph{generalized arithmetic sequence}, that is,
$$\Gamma = \<a, ah+d, ah+2d, \ldots, ah+kd\>$$
for some $a, h, d, k \in \ZZ_{\geq 1}$ with $\gcd(a,d) = 1$.  This family of numerical semigroups, introduced in~\cite{omidalirahmati}, are known for admitting concise characterizations of invariants that generally have high computational complexity in general (e.g., the Frobenius number).  

\begin{thm} \label{t:hwgenarith}
If $\Gamma$ is generated by a generalized arithmetic sequence, then the Huneke-Wiegand conjecture holds for any 2-generated monomial ideal in $\kk[\Gamma]$.
\end{thm}



\section{Visualizations for locating irreducible elements of $S_\Gamma^s$}
\label{sec:findingirreducibles}


Before giving the proof of Theorem~\ref{t:hwgenarith}, we demonstrate the utility of certain visuals that arose in obtaining this result.  
For a given numerical semigroup $\Gamma$, we may use the \texttt{Sage}~\cite{sage} package \texttt{LeamerMonoid}~\cite{leamersage} to compute, for each $s \notin \Gamma$, the set of all irreducible elements $(n,2)$.  A particularly helpful graphic emerges when we plot a point at $(s,n)$ if $(n,2) \in S_\Gamma^s$ is irreducible; Figures~\ref{f:3gen} and~\ref{f:genarith} each contain two examples.  Thus, the semigroup algebra $\kk[\Gamma]$ satisfies the Huneke-Wiegand conjecture for all 2-generated monomial ideals if, for each positive $s \notin \Gamma$, there exists at least one point $(s,n)$ in that column.  

\begin{figure}
\begin{center}
\includegraphics[width=2.8in]{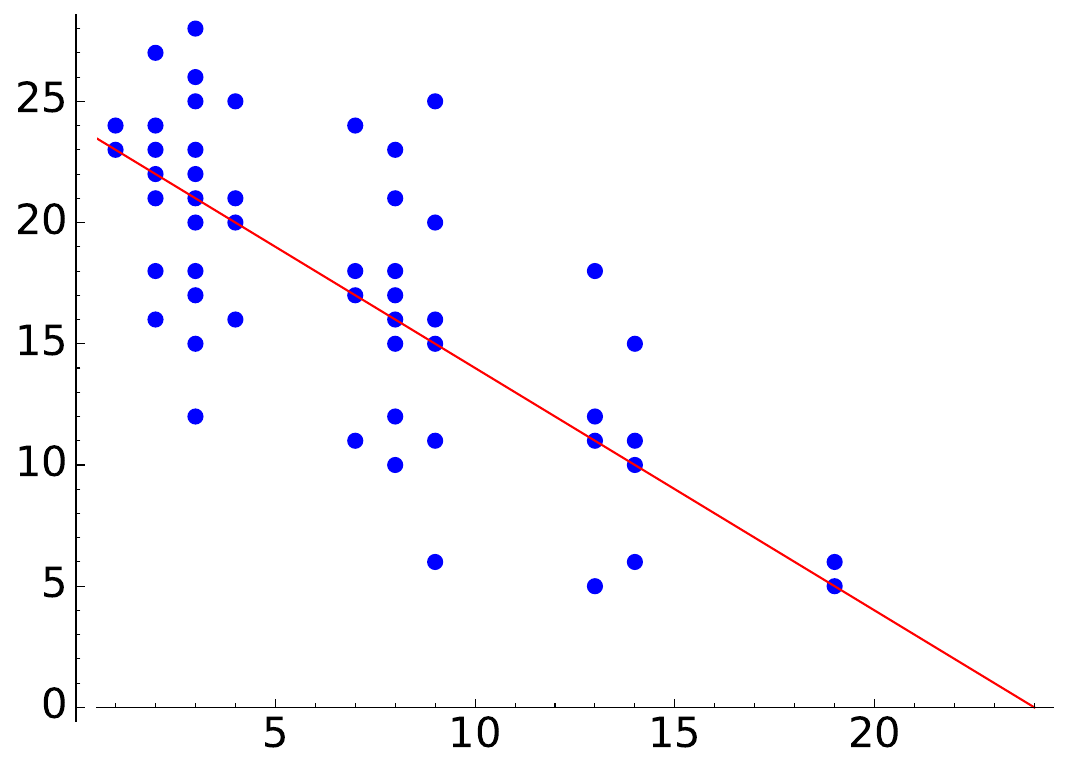}
\hspace{0.2in}
\includegraphics[width=2.8in]{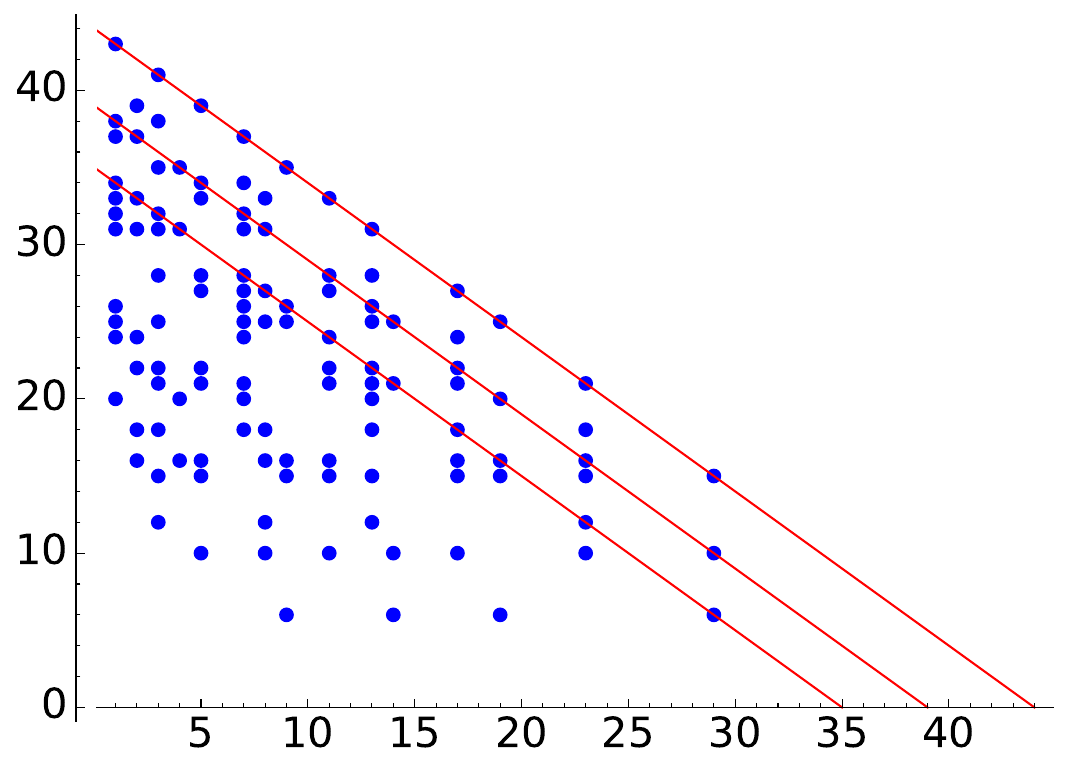}
\end{center}
\caption{The points $(s,n)$ for which $(n,2)$ is irreducible in $S_\Gamma^s$, where $\Gamma = \<5, 6\>$ (left) and $\Gamma = \<6, 10, 15\>$ (right).}
\label{f:3gen}
\end{figure}

Examining several of these graphs reveals a method of finding the requisite irreducible elements.  
For example, if $\Gamma = \<n_1, n_2\>$, such as in the left-hand graphic of Figure~\ref{f:3gen}, then $(\frob(\Gamma) + n_1 - s, 2) \in S_\Gamma^s$ is irreducible for each $s \notin \Gamma$; this is depicted with a diagonal red line defined by the equation $n = \frob(\gamma) + n_1 - s$ that contains a point in every column.



\begin{lemma}\label{l:gdivs}
Fix a symmetric numerical semigroup $\Gamma = \<n_1, \ldots, n_k\>$ with $s \notin \Gamma$.  Fix a generator $n_j$, and let $g = \gcd(\{n_1, \ldots, n_k\} \setminus \{n_j\})$.  We have $(\frob(\Gamma) - s + n_j, 2) \in S_\Gamma^s$, and if this element is reducible in $S_\Gamma^s$, then $g \mid s$.  
\end{lemma}

\begin{proof}
Since $\Gamma$ is symmetric, $\frob(\Gamma) - s \in \Gamma$, and since 
$$\frob(\Gamma) + s + n_j > \frob(\Gamma) + n_j > \frob(\Gamma),$$
both $\frob(\Gamma) + n_j$ and $\frob(\Gamma) + s + n_j$ lie in $\Gamma$.  This proves the first claim.  For the second claim, suppose 
$$(\frob(\Gamma) - s + n_j, 2) = (y, 1) + (\frob(\Gamma) - s + n_j - y, 1)$$
with $(y, 1), (\frob(\Gamma) - s + n_j - y, 1) \in S_\Gamma^s$.  Since $y, y + s \in \Gamma$, we conclude $y - n_j, y + s - n_j \notin \Gamma$ since $\Gamma$ is symmetric.  It must be that no expression for $y$ or $y + s$ as a sum of generators involves the generator $n_j$.  In particular, $y, y + s \in \<n_1, \ldots, \hat n_j, \ldots, n_k\>$, meaning $g \mid y$ and $g \mid y + s$, from which we conclude $g \mid s$.  
\end{proof}

Lemma~\ref{l:gdivs} makes quick work of the case $\Gamma = \<n_1,n_2\>$.  Indeed, in addition to verifying $(\frob(\Gamma) - s + n_1, 2) \in S_\Gamma^s$, Lemma~\ref{l:gdivs} implies if this element were reducible, then $n_2 \mid s$, which is impossible since $s \notin \Gamma$.  

For 3-generated numerical semigroups $\Gamma = \<n_1, n_2, n_3\>$, one can see by inspection of the right-hand graphic of Figure~\ref{f:3gen} that, unlike the 2-generated case above, there is no single line that contains a point in every column.  However, the 3 diagonal red lines depicted therein, each of which has the form $n = \frob(\gamma) - s + n_j$ for some~$j$, together contain at least one point in each column.  
These observations yield a relatively straightforward proof, included below, that Conjecture~\ref{conj:hw} holds in the case where $\Gamma$ has at most 3 generators; however, note that this case also follows from~\cite{hwintersection} since any such numerical semigroup is complete intersection.




\begin{prop}\label{p:3gensym}
Given any symmetric numerical semigroup $\Gamma = \<n_1, n_2, n_3\>$ and any $s \notin \Gamma$, the element $(\frob(\Gamma) - s + n_j, 2) \in S_\Gamma^s$ is irreducible for some $j$.
\end{prop}

\begin{proof}
Let $g_j = \gcd(\{n_1, n_2, n_3\} \setminus \{n_j\})$ for each $j$.  By Lemma~\ref{l:gdivs}, $(\frob(\Gamma) - s + n_j, 2) \in S_\Gamma^s$ for each generator $n_j$, and in order to prove one of these is irreducible in $S_\Gamma^s$, it suffices to assume $\lcm(g_1, g_2, g_3) \mid s$.  

Since $\Gamma$ is symmetric, \cite[Theorem~9.6]{numerical} implies that, after rearranging $n_1, n_2, n_3$ as needed, $d = \gcd(n_1, n_2) > 1$ and 
\begin{equation}\label{swap2}
dn_3 = a n_1 + b n_2 
\end{equation}
for some $a, b \in \ZZ_{\ge 0}$.  We claim $(\frob(\Gamma) - s + n_1, 2) \in S_\Gamma^s$ is irreducible.  Indeed, if this element were reducible, then it could be written as a sum
$$(\frob((\Gamma) - s + n_1, 2) = (\frob(\Gamma) - s + n_1 - z, 1) + (z, 1)$$
of atoms in $S_\Gamma^s$.  In particular, $z, z + s \in \Gamma$ since $(z, 1) \in S_\Gamma^s$, whereas $z - n_1, z + s - n_1 \notin \Gamma$ since $\Gamma$ is symmetric.  As such, any expression of $z$ and $z + s$ as a sum of generators must only involve $n_2$ and $n_3$, meaning
\begin{displaymath}
z = a_2n_2 + a_3n_3 \qquad \text{ and } \qquad z + s = b_2n_2 + b_3n_3
\end{displaymath}
for some $a_2, a_3, b_2, b_3 \in \ZZ_{\ge 0}$.  Moreover, we must have $0 \le a_3, b_3 < d$, as otherwise~\eqref{swap2} would yield an expression involving $n_1$.  Subtracting $z + s$ and $z$, we find 
$$s = (z + s) - z = (b_2 - a_2)n_2 + (b_3 - a_3)n_3$$
and since $d \mid s$, $d \mid n_2$ and $d \nmid n_3$, we conclude $d \mid (b_3 - a_3)$.  However, $|b_3 - a_3| < d$. Thus $b_3 - a_3 = 0$ and $s = (b_2 - a_2)n_2$, a contradiction.
\end{proof}

\section{Numerical monoids generated by generalized arithmetic sequences}
\label{sec:genarith}

We now turn our attention back to numerical semigroups $\Gamma$ generated by generalized arithmetic sequences; Figure~\ref{f:genarith} shows plots of irreducible elements of $S_\Gamma^s$ for two such semigroups.  Like the $3$-generated case, multiple lines are required to find an irreducible element for each $s \in \mathbb Z_{\geq 1} \setminus \Gamma$.  In particular, the combination of the line $n = \frob(\Gamma) + d - s$ and the horizontal line $n = ah + d$ provide the requisite irreducible elements of $S_\Gamma^s$.  These lines manifest in the form of Proposition~\ref{p:genarith}, which we prove after a short lemma.


\begin{lemma}\label{l:genarith}
Suppose $\Gamma = \<a, ah + d, ah + 2d, \ldots, ah + kd\>$ is a symmetric numerical semigroup with $3 \le k < a$ and $\gcd(a,d) = 1$, and fix $s \notin \Gamma$.  
We~have $\frob(\Gamma) - s + d \notin \Gamma$ if and only if for some $m \in \{0, \ldots, h-1\}$, 
$$s - d \equiv am \bmod (ah + kd).$$
\end{lemma}

\begin{proof}
First, suppose $s - d \equiv am \bmod (ah + kd)$ for some $m$ as above.  Fixing $l \in \ZZ$ such that $s - d = am + l(ah + kd)$ and noting that $l \ge 0$ since $s > 0$, we can write 
$$\frob(\Gamma) - s + d = \frob(\Gamma) - (d + am + l(ah + kd)) + d = \frob(\Gamma) - (am + l(ah + kd)).$$
Since $\Gamma$ is symmetric and $am + l(ah + kd)\in \Gamma$, this implies $\frob(\Gamma) - s + d \notin \Gamma$.  

Conversely, suppose $\frob(\Gamma) - s + d \notin \Gamma$.  Since $\Gamma$ is symmetric, $s - d \in \Gamma$, so suppose 
$$s - d = z_0a + z_1(ah + d) + \cdots + z_k(ah + kd).$$
We claim  (i) $z_j = 0$ for each $0 < j < k$, and (ii) $0 \le z_0 \le h-1$. 
Indeed, if $z_j > 0$ for $0 < j < k$, then
$$\begin{array}{r@{}c@{}l@{}r@{}l}
s = s - d + d
&{}={}& z_0a + \cdots +{} &    z_j&(ah + jd) + \cdots + z_k(ah + kd) + d \\[0.05in]
&{}={}& z_0a + \cdots +{} &(z_j-1)&(ah + jd) + \cdots + z_k(ah + kd) + (ah + (j+1)d),
\end{array}$$
which is impossible since $s \notin \Gamma$, and if $z_0 \ge h$, then
$$\begin{array}{r@{}c@{}r@{}l}
s = s - d + d
&{}={}&       z_0&a + \cdots + z_k(ah + kd) + d \\[0.05in]
&{}={}& (z_0 - h)&a + \cdots + z_k(ah + kd) + (ah + d),
\end{array}$$
which is again impossible since $s \notin \Gamma$.  Consequently, $s - d = z_0a + z_k(ah + kd)$, thereby completing the proof with $m = z_0$.
\end{proof}

\begin{figure}[t]
\begin{center}
\includegraphics[width=2.8in]{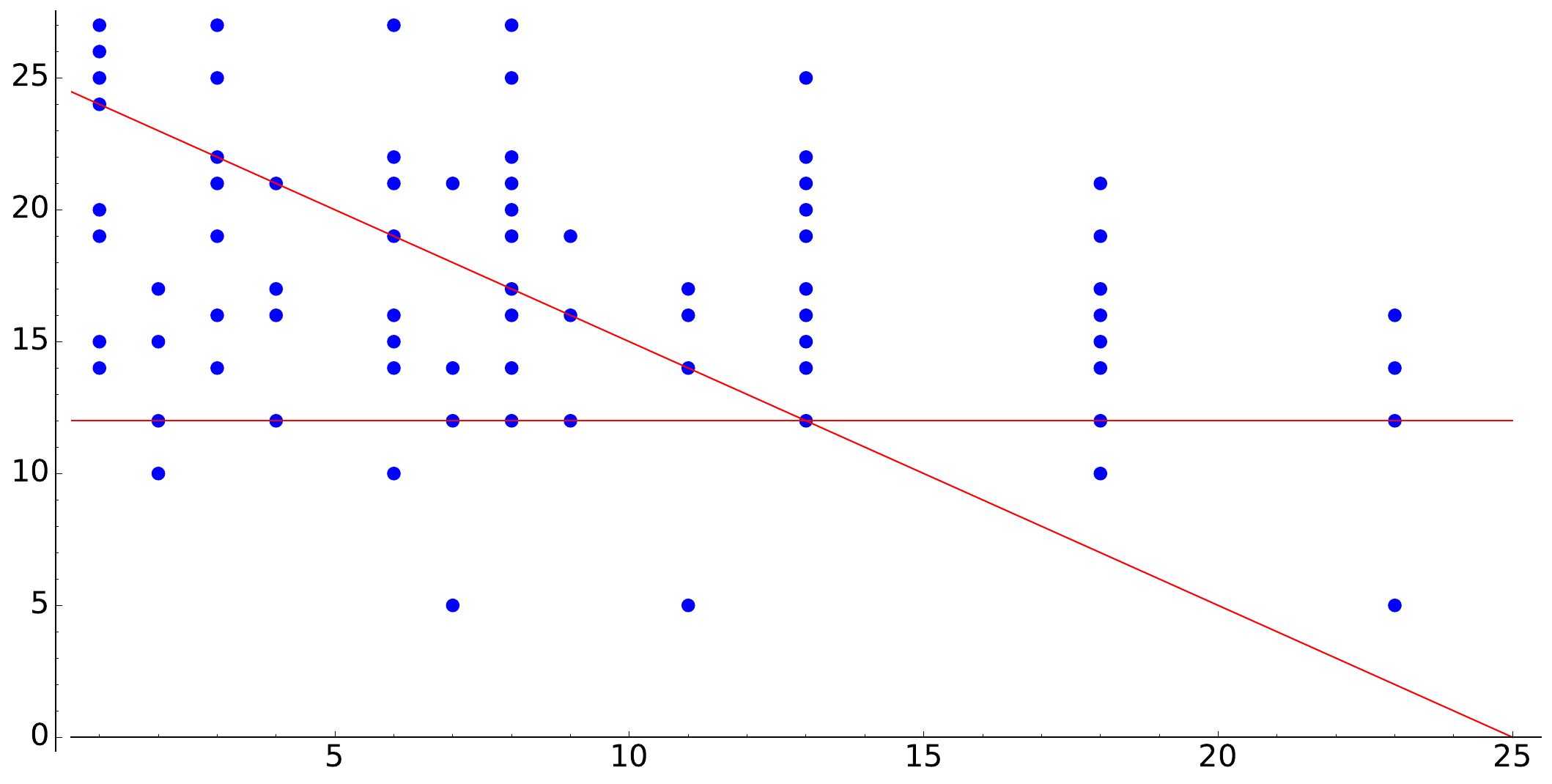}
\hspace{0.2in}
\includegraphics[width=2.8in]{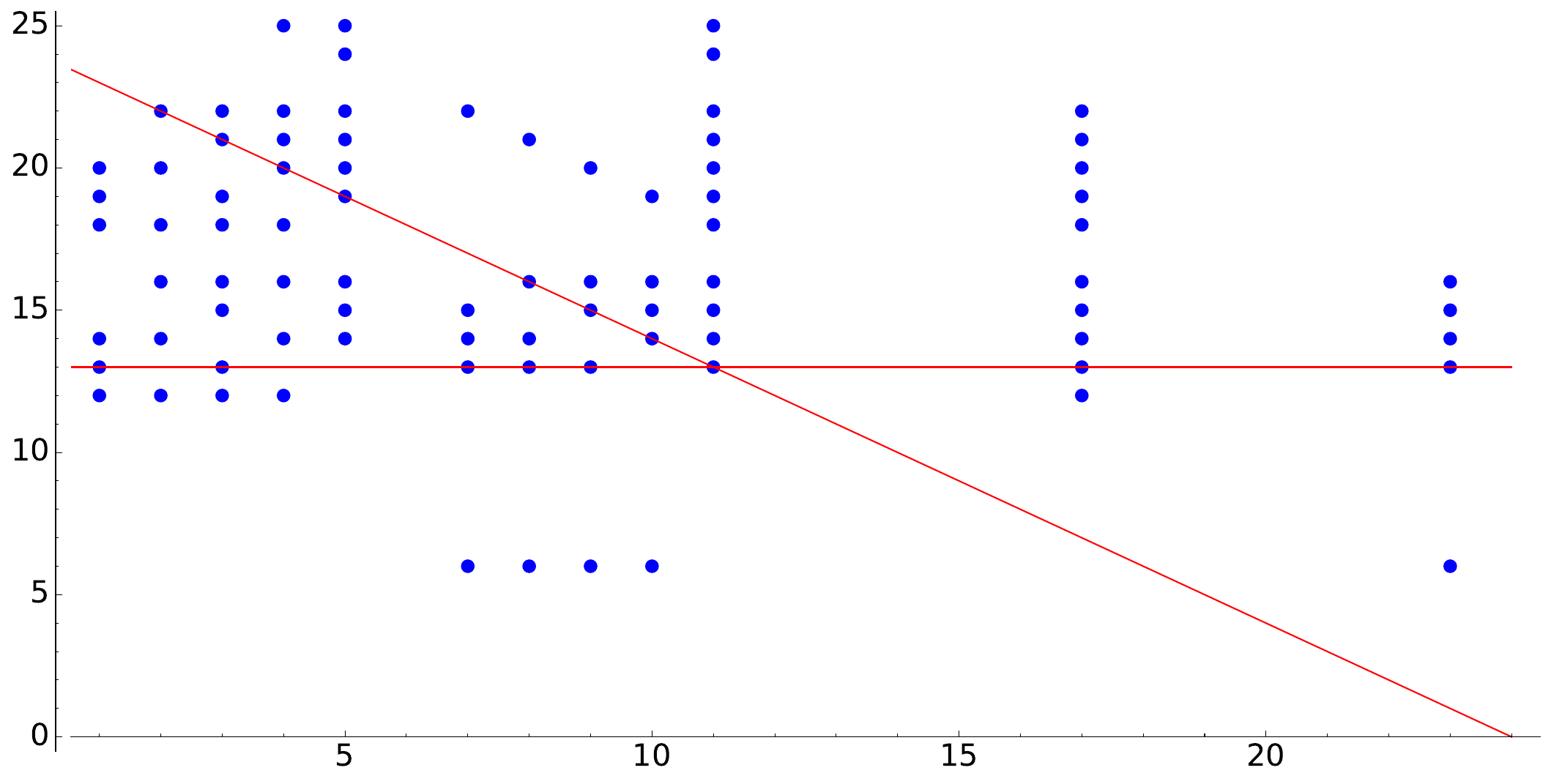}
\end{center}
\caption{The points $(s,n)$ for which $(n,2)$ is irreducible in $S_\Gamma^s$, where $\Gamma = \<5, 12, 14, 16\>$ (left) and $\Gamma = \<6, 13, 14, 15, 16\>$ (right) are each generated by a generalized arithmetic sequence.}
\label{f:genarith}
\end{figure}

\begin{prop}\label{p:genarith}
Suppose $\Gamma = \<a, ah + d, ah + 2d, \ldots, ah + kd\>$ is a symmetric numerical semigroup with $3 \le k < a$ and $\gcd(a,d) = 1$, and fix $s \notin \Gamma$.  
\begin{enumerate}[(a)]
\item 
If $\frob(\Gamma) - s + d \in \Gamma$, then $(\frob(\Gamma) - s + d, 2)$ is irreducible in $S_\Gamma^s$.  

\item 
If $\frob(\Gamma) - s + d \notin \Gamma$, then $(ah + d, 2)$ is irreducible in $S_\Gamma^s$.  

\end{enumerate}
So, the Huneke-Wiegand conjecture holds for any 2-generated monomial ideal in $\kk[\Gamma]$.  
\end{prop}

\begin{proof} 
For part~(a), suppose $\frob(\Gamma) - s + d \in \Gamma$.  Since 
$$\frob(\Gamma) + s + d > \frob(\Gamma) + d > \frob(\Gamma),$$
we also have $\frob(\Gamma) + d \in \Gamma$ and $\frob(\Gamma) + d + s \in \Gamma$.  This means $(\frob(\Gamma) - s + d, 2) \in S^s_\Gamma$.

Now, by way of contradiction, suppose $(\frob(\Gamma) - s + d, 2)$ is reducible, so that
$$(\frob(\Gamma) - s + d, 2) = (\frob(\Gamma) - s + d - n, 1) + (n, 1)$$
for some $(n, 1), (\frob(\Gamma) - s + d - n, 1) \in S_\Gamma^s$.  In particular, this means $z$ and $\frob(\Gamma) + d - n$ both lie in $\Gamma$, and since $\Gamma$ is symmetric, $n - d \notin \Gamma$.  We claim $n = z_0a$ for some $z_0 \in \ZZ_{\ge 0}$.  Indeed, if $n = z_0a + \cdots + z_k(ah + kd)$ with $z_1 > 0$, then
\begin{displaymath}
	n - d = (z_0 + h)a + (z_2 - 1)(ah + d) + \cdots + z_k(ah + kd),
\end{displaymath}
and if $z_j > 0$ for some $j > 1$, then
\begin{displaymath}
	n - d = z_0a + \cdots + (z_{j-1} + 1)(ah + (j-1)d) + (z_j - 1)(ah + jd) + \cdots + z_k(ah + kd),
\end{displaymath}
both of which are impossible since $n - d \notin \Gamma$.  By similar reasoning, $n + s \in \Gamma$ and $n + s - d \notin \Gamma$, so $n + s = z_0'a$ for some $z_0' \in \ZZ_{\ge 0}$.  This yields $s = (n + s) - n = (z_0' - z_0)a$, which is impossible since $s \notin \Gamma$.  As such, we conclude $(\frob(\Gamma) - s + d, 2)$ is irreducible in $S_\Gamma^s$, thereby proving part~(a).  

For part~(b), suppose $\frob(\Gamma) - s + d \notin \Gamma$.  By Lemma~\ref{l:genarith}, $s - d \equiv am \bmod (ah + kd)$ for some $m \in \{0, \ldots, h-1\}$, so let $l \in \ZZ$ with $s - d = am + l(ah + kd)$.  Since $ah + d \in \Gamma$ and $k \ge 3$, we have
$$\begin{array}{r@{}c@{}l}
ah + d + s &{}={}& (ah + 2d) + \phantom{2}am + \phantom{2}l(ah + kd) \in \Gamma \text{ and } \\[0.05in]
ah + d + 2s &{}={}& (ah + 3d) + 2am + 2l(ah + kd) \in \Gamma,
\end{array}$$
meaning $(ah + d, 2) \in S_\Gamma^s$.
Lastly, suppose by way of contradiction that $(ah + d, 2)$ is reducible in $S_\Gamma^s$, so that
$$(ah + d, 2) = (ah + d - n, 1) + (n, 1)$$
for some $(n, 1), (ah + d - n, 1) \in S_\Gamma^s$.  This means $n$ and $ah + d - n$ are both in $\Gamma$, but since both are less than $ah + d$, there exists $z_0, z_0' \in \ZZ_{\ge 0}$ such that $n = z_0a$ and $ah + d - n = z_0'a$.  This implies $ah + d = (z_0 + z_0')a$, which is impossible since $\gcd(a,d) = 1$.  This completes the proof.  
\end{proof}

\begin{proof}[Proof of Theorem~\ref{t:hwgenarith}]
If $\Gamma$ has at most 3 generators, then $\Gamma$ is complete intersection by~\cite[Corollary~10.5]{numerical}, so apply \cite[Corollary~22]{hwintersection}.  Otherwise, apply Proposition~\ref{p:genarith}.  
\end{proof}


\end{document}